\documentclass[11pt,reqno]{amsart}
\usepackage{amssymb,mathrsfs,mathtools}
\mathtoolsset{showonlyrefs}
\usepackage{hyperref}
\usepackage[shortlabels]{enumitem}
\usepackage{geometry}

\usepackage{tikz,lipsum,lmodern}
\usepackage[most]{tcolorbox}
\usepackage{xcolor}

\numberwithin{equation}{section}

\newtheorem{mainthm}{Theorem}
\newtheorem{thm}{Theorem}[section]
\newtheorem{cor}[thm]{Corollary}
\newtheorem{lem}[thm]{Lemma}

\theoremstyle{remark}

\def\bR {\mathbb{R}}
\def\bS {\mathbb{S}}
\def\bZ {\mathbb{Z}}

\def\grad {{\nabla}}

\newcommand{\ud}{\mathrm{d}}

\newcommand{\inte}{\operatorname{int}}

\newcommand{\veps}{\varepsilon}

\title{Topological Half-Space Theorems for Minimal Hypersurfaces in $\mathbb{R}^4$}
\author{Shrey Aryan}
\author{Alexander D. McWeeney}

\begin{document}
\begin{abstract}
The three-dimensional catenoid in $\mathbb{R}^4$ is a complete embedded minimal hypersurface contained in a slab, showing that the half-space theorem does not extend directly to higher dimensions. We show that this obstruction is topological in $\mathbb{R}^4$. More precisely, we prove that a connected, complete embedded minimal hypersurface $\Sigma^3\subset\mathbb{R}^4$ contained in a half-space with bounded curvature and trivial second homology must be a hyperplane. We obtain further rigidity when a connected, complete embedded minimal hypersurface is confined to a slab.  In particular, we show that if in addition the minimal hypersurface is proper and has trivial second homology then it must be a hyperplane, thus obtaining flatness without the bounded curvature assumption. If instead the minimal hypersurface has bounded curvature and finite second Betti number then it must have finite total curvature or finite index. As a corollary, we also obtain a new characterization of the three-dimensional catenoid as a minimal hypersurface confined in a slab with bounded curvature and second Betti number equal to one.
\end{abstract}
\maketitle
\section{Introduction}
The strong half-space theorem of Hoffman and Meeks
\cite{hoffmanMeeks1990} states that a connected, proper minimal surface contained in a half-space of $\mathbb{R}^3$ is flat. In the same work, they also noted that a higher-dimensional generalization of this result fails since the three-dimensional catenoid $\mathcal{C}^3\subset\mathbb{R}^4$ lies in a slab. However, the neck of the catenoid contributes nontrivial second homology, i.e., \(H_2(\mathcal C^3;\mathbb Z_2)\neq0\). It is therefore natural to ask whether a minimal hypersurface $\Sigma^3\subset\mathbb{R}^4$ confined to a half-space with trivial second homology \(H_2(\Sigma;\mathbb Z_2)=0\) is flat. To this end, we show that
\begin{mainthm}\label{thm: maintheorem}
Let \(\Sigma^3 \subset \mathbb{R}^4\) be a complete connected embedded minimal hypersurface with bounded curvature and \(H_2(\Sigma; \mathbb{Z}_2) = 0\). If \(\Sigma\) is contained in a half-space, then \(\Sigma\) is a plane.
\end{mainthm}
Although we assume that the hypersurface has bounded curvature, we do not assume properness, unlike in the strong half-space theorem \cite{hoffmanMeeks1990}. Properness of minimal hypersurfaces with bounded curvature and finite second Betti number $\operatorname{dim}H_2(\Sigma;\mathbb{Z}_2)<\infty$ was recently proved in \cite{aryanMcWeeneyTinaglia2026properness}.

When the hypersurface is confined to a slab, we can prove considerably more. First, we can extend the conclusions of Theorem~\ref{thm: maintheorem} without assuming bounded curvature:
\begin{mainthm}\label{thm:slab-without-curvature}
Let $\Sigma^3\subset\bR^4$ be a complete, connected, properly embedded
minimal hypersurface without boundary with $H_2(\Sigma;\bZ_2)=0$.
If $\Sigma$ is contained in a slab, then $\Sigma$ is a plane.
\end{mainthm}
Next, we show that an embedded minimal hypersurface contained
in a slab with finite topology and bounded curvature has finite
total curvature or finite index.
\begin{mainthm}\label{thm:main2}
Let \(\Sigma^3 \subset \mathbb{R}^4\) be a complete, connected, embedded minimal surface. Then \(\Sigma\) has finite total curvature if and only if \(\Sigma\) is contained in a slab, has bounded curvature, and \(H_2(\Sigma; \mathbb{Z}_2)\) is finite dimensional.
\end{mainthm}
Note that the forward implication follows from \cite{schoen1983uniqueness} and by adapting the ideas of \cite{li2017index}. See also Appendix A in \cite{chodosh2026minimal} for more details. Theorem~\ref{thm:main2} is also a step toward a higher dimensional version of a result of Collin's \cite[Theorem 2]{collin1997topologie}. Collin shows that a properly embedded complete minimal surface in \(\bR^3\) with at least two ends has finite topology if and only if it has finite total curvature. Notably, while we assume bounded curvature and containment in a slab, we need no assumption on \(H_1(\Sigma; \bZ_2)\). As a corollary of Theorem~\ref{thm:main2}, we obtain a new characterization of the three-dimensional catenoid,
\begin{cor}\label{cor:catenoid}
Let $\Sigma^3\subset\bR^4$ be a complete, connected, embedded
minimal hypersurface, contained in a slab and
with bounded curvature. Then $\Sigma$ is the three-dimensional
catenoid up to rigid motions and dilations if and only if
$\dim_{\bZ_2}H_2(\Sigma;\bZ_2)=1$.
\end{cor}
Earlier characterizations of the catenoid were given under the assumption of rotational symmetry \cite{choeHoppe2013catenoidHelicoid}, $2/3$-stability together with a curvature decay condition \cite{tam2009stability}, as a connected minimal hypersurface that has two ends regular at infinity \cite{schoen1983uniqueness} and more recently as an index one minimal hypersurface with finite total curvature \cite{chodosh2026minimal}. The above Corollary adds another characterization which is based on topology and confinement.

The early literature related to half-space theorems was motivated in part by the Calabi--Yau conjectures for complete minimal surfaces. Jorge and Xavier \cite{jorgeXavier1980} constructed a complete non-flat minimal surface contained between two parallel planes, showing that completeness alone does not prevent a minimal surface from being confined to a slab. Xavier's result \cite{xavier1984} showed it suffices to add the assumption of bounded curvature, and Hoffman and Meeks \cite{hoffmanMeeks1990} then showed it suffices to add the assumption of properness. Subsequent works have generalized
these results under various assumptions involving bounded curvature,
graphicality, or recurrence of either the minimal surface or the ambient manifold
\cite{bessaJorgeOliveiraFilho2001,rosenbergSchulzeSpruck2013,
colomboMagliaroMariRigoli2022,bessaJorgePessoa2024}.

In a related direction, Nadirashvili \cite{nadirashvili1996hadamard} constructed a complete bounded minimal immersion of a disk into $\mathbb{R}^3$, while Colding and Minicozzi \cite{coldingMinicozzi2008calabiYau} proved that complete embedded minimal surfaces in $\mathbb{R}^3$ with finite topology are proper. Their work relied on the structure theory of embedded minimal surfaces developed in \cite{coldingMinicozzi2004I,coldingMinicozzi2004II,
coldingMinicozzi2004III,coldingMinicozzi2004IV,coldingMinicozzi2015V}.
A key tool in this theory is the one-sided curvature estimate, which may be viewed as a local form of half-space rigidity and was later used by Meeks and Rosenberg \cite{meeksRosenberg2005uniqueness} in their classification of properly embedded simply connected minimal surfaces in $\mathbb R^3$.

Our motivation in the present work is to understand analogous questions for complete minimal hypersurfaces, where many of the classical two-dimensional tools, such as the Weierstrass representation,
parabolicity, and Gauss--Bonnet, are no longer available. In earlier work
\cite{aryanMcWeeney2026calabi}, we constructed examples of complete improperly embedded minimal hypersurfaces in
$\mathbb{R}^{n+1}$, $n\geq3$, contained in a slab. These examples have unbounded curvature and infinite topology, and show that additional hypotheses are needed in higher dimensions to resolve the Calabi--Yau conjectures. In the same work, we also showed that the closure of a complete, connected, non-proper embedded minimal hypersurface in $\bR^{n+1}$ for $3\leq n\leq 5$ with bounded curvature is a minimal lamination which is confined either to a slab or to a half-space. When $n=2$, Rosenberg used Xavier's half-space theorem to rule out this scenario to show that every complete minimal surface that is injectively immersed and has bounded curvature in $\bR^3$ is proper. However, no analog of this fact is known in higher dimensions. This motivated us to understand extensions of the half-space-type theorems for minimal hypersurfaces in higher dimensions starting with $\mathbb{R}^4$. 

Our work also takes inspiration from the recent work by Colding and Minicozzi on properly embedded minimal hypersurfaces in slabs \cite{colding2026minimal}. In particular, they prove that any such hypersurface has Euclidean volume growth and satisfies a weighted tilt estimate. When \(\dim \Sigma = 3\), this estimate takes the following form: if $h=x_4|_\Sigma$ is the height function, $\Sigma\subset\bR^3\times[0,1]$, and $|x|$ denotes the Euclidean distance from $x\in\Sigma$ to the origin, then
\begin{equation}
        \int_\Sigma
        \frac{|\grad_\Sigma h|^2}{|x|}
        \,\ud\mu_\Sigma
        <\infty.
\end{equation}
We interpret this result as a bound on how ``nongraphical'' a hypersurface $\Sigma$ in a slab can be. For $\Sigma$ in a half-space, we prove a similar integral bound for a subset of $\Sigma$ close to the boundary plane of the half-space.

Finally, we record some remarks about the assumptions in our main theorems. First, as we noted earlier, the topological assumption in Theorem~\ref{thm: maintheorem} cannot be removed, since the three-dimensional catenoid $\mathcal C^3\subset\mathbb{R}^4$ satisfies all the assumptions except that its second homology is nontrivial. Next, the natural generalizations of Theorems~\ref{thm: maintheorem} and \ref{thm:slab-without-curvature} for $\Sigma^n\subset \bR^{n+1}$ with vanishing $H_{n-1}(\Sigma^n;\bZ_2)=0$ are false since $\mathcal{C}^3\times \mathbb{R}\subset \bR^5$ has $H_3(\mathcal{C}^3\times \mathbb{R};\bZ_2)=0$ while $H_2(\mathcal{C}^3\times \mathbb{R};\bZ_2)\simeq \bZ_2$; therefore one would need to control more homology groups. The confinement to a half-space assumption also appears to be subtle. In particular, if $\mathcal H^2\subset\mathbb{R}^3$ denotes the standard helicoid, then $\mathcal H^2\times\mathbb{R}\subset\mathbb{R}^4$ is a complete, properly embedded, non-flat minimal hypersurface with bounded curvature and is diffeomorphic to $\mathbb{R}^3$. However, its convex hull is all of $\mathbb{R}^4$, so it is not contained in any slab or half-space, and its volume growth is quartic rather than cubic. Finally, embeddedness and bounded curvature cannot both be removed since by Nadirashvili's construction \cite{nadirashvili1996hadamard}, there is a complete bounded minimal immersion $\widetilde\Sigma^2\subset B_1(0)\subset\mathbb{R}^3$ with disk topology. Hence $\widetilde\Sigma^2\times\mathbb{R}$ is a complete minimal immersion,
diffeomorphic to $\mathbb{R}^3$, and contained in a slab in $\mathbb{R}^4$. This example is not properly embedded and has unbounded curvature. The most natural weakening of our assumptions would seem to be to allow unbounded curvature or to consider higher dimensional hypersurfaces, but it is not yet clear to us how to generalize in this direction. Besides this point, however, these examples show that our theorem isolates a true rigidity property of minimal hypersurfaces in \(\mathbb{R}^4\).
\subsection{Proof Sketch}
We first use the main theorem proved recently in \cite{aryanMcWeeneyTinaglia2026properness} to assume that $\Sigma$ is proper. After a rigid motion and a vertical translation, let
\begin{align*}
        h=x_4|_\Sigma\geq0,
        \quad
        \inf_\Sigma h=0.
\end{align*}
If $h$ vanishes at a point, the strong minimum principle gives the result. Thus, we may assume that $h>0$. Since $\inf_\Sigma h=0$, the function $h$ is nonconstant. Bounded curvature gives $|\grad_\Sigma h|^2\leq2\Lambda_*h$ for $\Lambda_*=\max\{1,\sup_\Sigma|A_\Sigma|\}<\infty$. Choose a sufficiently small regular value $\tau\in(0,\min\{1,\sup_\Sigma h\})$. We then show that every connected component $V$ of $\{h<\tau\}$ is a global graph
\begin{align*}
        V=\{(y,u(y)):y\in D\},
        \qquad
        u=\tau\text{ on }\partial D,
        \qquad
        |Du|<1.
\end{align*}
Here $D=\pi(V)$ is the vertical projection of $V$ on the plane $\{x_4=0\}.$ Choose $V$ which meets $\{h<\tau/4\}$. Since $H_2(\Sigma;\mathbb Z_2)=0$, every connected component of a regular height level is noncompact. Thus, for a regular value $t\in(\tau/4,\tau/2)$, there are connected noncompact components
\begin{align*}
        S\subset\partial V\cap\{h=\tau\},
        \quad
        \Gamma\subset V\cap\{h=t\}.
\end{align*}
Extend $g=\tau-u$ by zero outside $D$. On $V$, let $x:\Sigma\to\bR^4$ denote the position vector and set
\begin{align*}
        E=|\grad_\Sigma h|^2,
        \quad
        \Phi=\frac12(\tau-h)^2,
        \quad
        \psi=(1+|x|^2)^{-1/2}.
\end{align*}
The identities $\Delta_\Sigma\Phi=E$, $\Delta_\Sigma\psi\leq0$, and the cubic volume bound for the single graph $V$ give
\begin{align*}
        \int_{\mathbb R^3}
        \frac{|\grad g|^2}{1+|y|}
        \,\ud y
        <\infty.
\end{align*}
However, the projected sets $\pi(S)$ and $\pi(\Gamma)$ are connected and unbounded, and $g$ takes two different constant values on them. Thus the projected sets meet every sufficiently large sphere in $\bR^3$. The Lipschitz bound on $g$ gives two fixed-size caps on each such sphere where the values of $g$ are uniformly separated. An estimate for the Dirichlet energy of $g$ restricted to these spheres yields a lower bound for the integral in the preceding display and shows that it diverges, a contradiction. Hence $h$ vanishes somewhere, and the strong maximum principle implies that $\Sigma$ is a hyperplane.

To prove Theorem~\ref{thm:slab-without-curvature}, we use the cubic volume bound and weighted tilt estimate of Colding and Minicozzi \cite{colding2026minimal}. After normalizing the slab, for $p\in \bR^4$ define
\begin{align*}
        F(p)=\int_\Sigma\frac{|\grad_\Sigma h(x)|^2}{1+|x-p|}\,\ud\mu_\Sigma(x).
\end{align*}
The weighted tilt estimate implies that $\inf_\Omega F=0$ on every connected unbounded set $\Omega\subset \bR^4$. This is because a positive lower bound on $F$ would give a lower bound of order $1/\log R$ for the weighted tilt integral on each annulus $\{R<|x|<6R\}$, contradicting its finiteness after summing over disjoint annuli.

Suppose that $h$ is nonconstant. Since $H_2(\Sigma;\bZ_2)=0$, every connected component of a regular height level is noncompact, and hence unbounded by properness. Fix a regular value $t\in\inte h(\Sigma)$ and a smooth function $\psi$ with $0\leq\psi\leq1$ which equals one near $t$. A weighted mean value inequality gives
\begin{align*}
        \liminf_{R\rightarrow\infty}\frac1{R^3}
        \int_{\Sigma\cap B_R(0)}\psi(h)\,\ud\mu_\Sigma
        \geq\omega_3-C_\psi F(p)
        \qquad\text{for }p\in\{h=t\},
\end{align*}
where $\omega_3$ is the volume of the unit ball in $\bR^3$. Taking the infimum of $F$ along one component of $\{h=t\}$ shows that the left-hand side is at least $\omega_3$. We can choose arbitrarily many distinct regular values and corresponding functions $\psi$ with disjoint supports. Summing the resulting lower bounds contradicts the cubic volume bound. Thus $h$ is constant and $\Sigma$ is a plane.

To prove Theorem~\ref{thm:main2}, we first show that a level set of the height function \(\{h = t\}\) can have at most \(\dim H_2(\Sigma; \bZ_2)\) connected compact components, using that \(h\) is harmonic on \(\Sigma\) and the strong maximum principle. Then we show that the set \(Z := \{\langle \nu, e_4\rangle = 0\}\) must be compact using the weighted tilt integral and the uniform volume estimate \cite[Theorem 0.8]{colding2026minimal}
\begin{equation}
    (1 - Cr^{-2})N \le \frac{\mu_\Sigma(B_r \cap \Sigma)}{\omega_3 r^3} \le N
\end{equation}
for \(r > R_0\) with \(R_0\) and \(C\) constants depending only on dimension, and an integer \(N\) depending on \(\Sigma\). The weighted tilt integral will let us rule out the possibility that \(Z\) has noncompact components; then, assuming for contradiction \(Z\) has infinitely many compact components, the volume bound combined with bounded curvature will let us find a level set of \(\{h = t\}\) (for some \(t\)) with infinitely many compact components. This will contradict finite second homology. Finally, if $Z$ is compact, then properness of $\pi|_\Sigma$ shows that $\Sigma$ is graphical at infinity. The bounded heights give bounded slopes by interior gradient estimates, so \cite[Proposition~3]{schoen1983uniqueness}
implies regularity at infinity and hence finite total curvature.

\subsection{Acknowledgments}
The authors thank their advisors Tobias Colding and William Minicozzi for their commentary and support. The first author acknowledges support from the Simons Dissertation Fellowship and NSF DMS Grant 2405393. The first author thanks Jiwon Park for suggesting characterizing the catenoid as in Corollary~\ref{cor:catenoid}. The second author acknowledges support from the National Science Foundation.

\section{Preliminary Topological Facts}
In this section we prove a topological separation theorem for homologically trivial submanifolds. This will in particular show that the height function cannot have many compact level sets using the separation and the maximum principle.

\begin{lem}\label{lem:one-end-compact-side}
Let $M^3 \subset \mathbb{R}^4$ be an orientable connected noncompact smooth manifold without boundary. If $\Gamma\subset M$ is a compact connected smoothly embedded two-sided surface without boundary which is homologically trivial (i.e., \([\Gamma] = 0\)), then
\begin{align}
        M\setminus\Gamma=U_\Gamma\sqcup V_\Gamma,
\end{align}
where $U_\Gamma$ is relatively compact and $V_\Gamma$ is not relatively compact. In particular, $\overline U_\Gamma$ is a compact smooth domain with boundary $\Gamma$. Furthermore, if we suppose that \(H_2(M; \bZ_2) = 0\), then \(M\) has one end.
\end{lem}

\begin{proof}
Suppose for contradiction that \(\Gamma\) does not separate \(M\).
Let \(\alpha\) be a path in \(M\setminus\Gamma\) joining the opposite
sides of a tubular neighborhood of \(\Gamma\). Since
\([\Gamma]=0\) in \(H_2(M;\bZ_2)\), there is a singular \(3\)-chain
\(C\) with \(\bZ_2\)-coefficients such that \(\partial C\) is the
fundamental cycle of \(\Gamma\). Choose a compact connected smooth
domain \(N\subset M\) whose interior contains
\(\operatorname{spt}C\), the path \(\alpha\), and the closure of this
tubular neighborhood. Then \(C\) is a \(3\)-chain in \(N\) whose
support avoids \(\partial N\), and hence
\[
        [\Gamma]=0
        \quad\text{in}\quad
        H_2(N,\partial N;\bZ_2).
\]
The proof of
\cite[Proposition~1.7.17]{martelliGeometricTopology}, applied with
\(\bZ_2\)-coefficients, shows that \(\Gamma\) separates \(N\).
Since \(N\) and \(\Gamma\) are connected and \(\Gamma\) is two-sided,
\(N\setminus\Gamma\) has exactly two connected components, and the
opposite sides of the tubular neighborhood lie in different
components. This contradicts the existence of
\(\alpha\subset N\setminus\Gamma\).

Thus \(\Gamma\) separates \(M\). Since \(M\) and \(\Gamma\) are
connected and \(\Gamma\) is two-sided, \(M\setminus\Gamma\) has
exactly two connected components. Since \(M\) is noncompact, at least
one of these components is not relatively compact. Denote such a
component by \(V_\Gamma\) and denote the other by \(U_\Gamma\).

Choose a locally finite triangulation of $M$ adapted to $\Gamma$,
and identify $\Gamma$ with the mod-two sum of its $2$-simplices.
Since $[\Gamma]=0$, there is a simplicial $3$-chain
$T=\sum_\sigma a_\sigma\sigma$ with $\partial T=\Gamma$,
where $\sigma$ ranges over the tetrahedra of the triangulation,
$a_\sigma\in\mathbb Z_2$, and only finitely many $a_\sigma$
are nonzero. If two tetrahedra $\sigma,\sigma'$ share a triangular face
$\tau$ not contained in $\Gamma$, then the coefficient of
$\tau$ in $\partial T$ is $a_\sigma+a_{\sigma'}=0$.
Thus $a_\sigma=a_{\sigma'}$. Any two tetrahedra whose interiors
lie in the same component of $M\setminus\Gamma$ can be joined
by a finite sequence of tetrahedra sharing faces not contained
in $\Gamma$. Consequently, there are $a_U,a_V\in\mathbb Z_2$
such that $a_\sigma=a_U$ when
$\operatorname{int}\sigma\subset U_\Gamma$ and
$a_\sigma=a_V$ when
$\operatorname{int}\sigma\subset V_\Gamma$. The set $\operatorname{spt}T=\bigcup_{a_\sigma=1}\sigma$
is compact. Since $V_\Gamma$ is not relatively compact,
$a_V=0$. Along any triangular face contained in $\Gamma$,
the identity $\partial T=\Gamma$ gives $a_U+a_V=1$, so $a_U=1$.
Hence $U_\Gamma\subset\operatorname{spt}T$, and therefore
$\overline U_\Gamma\subset\operatorname{spt}T$ is compact.
Finally, the tubular neighborhood theorem shows that
$\overline U_\Gamma$ is a compact smooth domain with
boundary $\Gamma$.

To show that \(M\) has one end when \(H_2(M;\bZ_2) = 0\), we note that the proof of \cite[Lemma~2.1]{chodosh20263} applies with
$\bZ_2$-coefficients. This is because, if $M$ had more than one end, there would be a smooth relatively compact domain $K\Subset M$ such that $M\setminus\overline K$ has at least two connected components that are not relatively compact. Let $E$ be one of these components. Then both $E$ and $M\setminus\overline E$ are not relatively compact. As in the proof of \cite[Lemma~2.1]{chodosh20263}, this gives $[\partial E]\neq0$ in $H_2(M;\bZ_2)$, contradicting the assumption that \(H_2(M; \bZ_2) = 0\). Thus $M$ has one end.
\end{proof}

\begin{lem}\label{lem:compact-height-components}
Let $M$ be a connected smooth $3$-manifold and let $f$ be a
nonconstant harmonic function on $M$. If $t$ is a regular value of
$f$, then the homology classes of the compact connected components
of $\{f=t\}$ are linearly independent in $H_2(M;\bZ_2)$.
Consequently, if
\[
\dim_{\bZ_2}H_2(M;\bZ_2)=B<\infty,
\]
then $\{f=t\}$ has at most $B$ compact connected components.
\end{lem}

\begin{proof}
Suppose that $C_1,\ldots,C_k$ are distinct compact connected
components of $\{f=t\}$ and that a nonempty subcollection satisfies
\[
\sum_{i\in I}[C_i]=0
\quad\text{in }H_2(M;\bZ_2).
\]
Set $C=\bigcup_{i\in I}C_i$. Choose a locally finite triangulation
of $M$ for which $C$ is a subcomplex. The preceding relation is represented by a finite mod-two
simplicial $3$-chain $T=\sum_\sigma a_\sigma\sigma$ with
$\partial T=C$. As in the proof of
Lemma~\ref{lem:one-end-compact-side}, the coefficients
$a_\sigma$ agree for tetrahedra whose interiors lie in the
same component of $M\setminus C$. Let $\Omega$ be the union
of the components with coefficient one. Then
$\operatorname{spt}T=\overline\Omega$ is compact and
\[
\partial\Omega=C.
\]
Since $f=t$ on $\partial\Omega$, the maximum and minimum principles
give $f=t$ on $\Omega$. As $\Omega$ contains a nonempty open set,
unique continuation implies that $f\equiv t$ on $M$, a
contradiction.
\end{proof}

\begin{lem}\label{lem:no-compact-regular-level}
Let $M$ satisfy the assumptions of
Lemma~\ref{lem:one-end-compact-side}, and suppose that $M$
is minimal and $H_2(M;\bZ_2)=0$. Let \(h: M \rightarrow \bR\) be the height function. If $c$ is a regular value of $h$, then $\{h=c\}$ has no compact connected component.
\end{lem}

\begin{proof}
    As the height function on a minimal surface is harmonic, the result follows from Lemma~\ref{lem:compact-height-components} with \(B = 0\).
\end{proof}

\section{Preliminary Analytic Facts}
Our goal is now to analyze the height function \(h: \Sigma \rightarrow \bR\). In particular, we seek to find a large graphical subset of \(\Sigma\) and derive a weighted tilt integral on that subset analogous to the one obtained by Colding and Minicozzi \cite{colding2026minimal}. We begin by doing some normalizations of the surface.

Let $\Sigma$ be a complete connected embedded minimal hypersurface with bounded curvature contained in a half-space with finite second homology. By \cite{aryanMcWeeneyTinaglia2026properness}, $\Sigma$ is proper. Suppose that $\Sigma$ is not a hyperplane. Note that because \(\Sigma\) is complete, properly embedded, and has codimension 1, it is orientable \cite{samelson1969Orientability}. After a rigid motion, suppose that $\Sigma$ is contained in a half-space bounded by a horizontal hyperplane. Translating by $-\inf_\Sigma x_4$, set
\begin{align}
        h=x_4|_\Sigma.
\end{align}
Then the harmonicity of the coordinate functions and the maximum principle give
\begin{align}
        h> 0,
        \qquad
        \inf_\Sigma h=0,
        \qquad
        \Delta_\Sigma h=0.
        \label{eq:height-basic-properties}
\end{align}
Let
\begin{align}
        \pi:\mathbb R^4\rightarrow\mathbb R^3,
        \qquad
        \pi(x_1,x_2,x_3,x_4)=(x_1,x_2,x_3),
\end{align}
be the projection map and set
\begin{align}
        \Lambda_*=\max\{1,\sup_\Sigma|A_\Sigma|\}.
\end{align}
Since we assume \(\Sigma\) is not a hyperplane, \(h\) is nonconstant. Since \(\inf_\Sigma h=0\), Sard's theorem gives a regular value $\tau$ of \(h\) satisfying
\begin{align}
        0<\tau<\frac{1}{4}\min\left\{\sup_\Sigma h, \,\frac{1}{4\Lambda_*}\right\}.
        \label{eq:tau-choice}
\end{align}

We now claim that the connected components of the set \(\{h < \tau\}\) are graphs over \(\{x_4 = 0\}\) with bounded gradient, which we establish in the following lemmas. First we show that \(|\nabla_\Sigma h|\) is small when \(h\) is small; this essentially follows from the fact that when \(h\) is small, bounded curvature forces the tangent plane of \(\Sigma\) to be almost parallel to \(\{x_4 = 0\}\), or else \(\Sigma\) would not be in a half-space.

\begin{lem}\label{lem:height-gradient}
Let $\Sigma^3\subset\{x_4>0\}\subset\mathbb R^4$ be a complete
minimal hypersurface without boundary with bounded curvature.
Set $h=x_4|_\Sigma$ and
$\Lambda_*=\max\{1,\sup_\Sigma|A_\Sigma|\}$. Then
    \begin{align}
        |\nabla_\Sigma h|^2 \le 2\Lambda_* h.
        \label{eq:height-gradient-estimate}
    \end{align}
\end{lem}
\begin{proof}
Let $\nu$ be a local choice of unit normal on $\Sigma$, use the convention
\begin{align}
        A_\Sigma(X,Y)
        =
        \langle\nabla_XY,\nu\rangle,
\end{align}
and define $w=\langle e_4,\nu\rangle$. Then for \(X, Y \in T \Sigma\) we have
\begin{align}
    \nabla_\Sigma^2h(X, Y) & = \langle \nabla_X \nabla_\Sigma h, Y\rangle 
    \\ & = \langle \nabla_X e_4^T, Y \rangle
    \\ & = \langle \nabla_X e_4 - \nabla_X (w\nu), Y \rangle
    \\ & = -(\nabla_X w)\langle \nu, Y \rangle - w\langle \nabla_X \nu, Y\rangle
    \\ & = w \langle \nabla_X Y, \nu\rangle
    \\ & = wA_\Sigma(X, Y).
\end{align}
It now follows that
\begin{align}
        |\grad_\Sigma^2h|= |w A_\Sigma| \le \Lambda_*.
\end{align}
Fix $p\in\Sigma$ and set $a=|\grad_\Sigma h(p)|$. If $a>0$, let $\gamma$ be the complete unit-speed geodesic satisfying
\begin{align}
        \gamma(0)=p,
        \qquad
        \gamma'(0)=-\frac{\grad_\Sigma h(p)}{a}.
\end{align}
Since $h>0$, Taylor's formula and the estimate $|(h\circ\gamma)''|\leq\Lambda_*$ give
\begin{align}
        0
        <
        h\left(\gamma\left(\frac{a}{\Lambda_*}\right)\right)
        \leq
        h(p)-\frac{a^2}{2\Lambda_*}.
\end{align}
Therefore
\begin{align}
        |\grad_\Sigma h|^2
        \leq
        2\Lambda_*h,
\end{align}
as we wanted to show.
\end{proof}
Now we prove that the components of \(\{h < \tau\}\) are indeed graphs.

\begin{lem}\label{lem:low-component-global-graph}
Let $V$ be a connected component of $\{h<\tau\}$, and set
\begin{align}
        M=\overline V \cap \Sigma,
        \qquad
        D=\pi(V).
\end{align}
Here $\overline V$ denotes the closure of $V$ in $\Sigma$. Then $\pi:M\rightarrow\overline D$ is a diffeomorphism. In particular, there is $u\in C^\infty(\overline D)$ such that
\begin{align}
        V=\{(y,u(y)):y\in D\},
        \qquad
        u=\tau
        \quad\text{on }\partial D,
\end{align}
and
\begin{align}
        |Du|<1.
        \label{eq:low-component-slope}
\end{align}
\end{lem}

\begin{proof}
Since $V$ is a connected component of $\{h<\tau\}$, it is closed in $\{h<\tau\}$, and hence $M\cap\{h<\tau\}=V.$ Since $\tau$ is a regular value, $M$ is a connected smooth three-manifold with boundary satisfying
\begin{align}
        \inte M=V,
        \quad
        \partial M=M\cap\{h=\tau\}.
        \label{eq:M-boundary}
\end{align}
This is because, if $p\in M\setminus V$, regular-level coordinates centered at $p$ identify $\{h<\tau\}$ locally with a connected half-ball. Since this half-ball meets $V$, it is contained in $V$. For $p\in M$ and $v\in T_p\Sigma$, equations \eqref{eq:tau-choice} and \eqref{eq:height-gradient-estimate} give
\begin{align}
        |d\pi_p(v)|^2
        &=
        |v|^2-dh(v)^2\geq
        \left(1-|\grad_\Sigma h(p)|^2\right)|v|^2
        \geq
        \frac12|v|^2.
        \label{eq:vertical-projection-lower}
\end{align}
Thus $M$ is locally graphical over $\{x_4=0\}$, up to its boundary, and every local graphing function has gradient of norm less than one.

We now show that \(M\) is a global graph. We will use \cite[Lemma~1.4]{meeksRosenberg2005uniqueness} to show that \(\pi|_{M}\) is injective. To this end we need to verify that \(\pi|_{M}\) is a proper submersion and that $\pi$ is injective for each component of $\partial M.$ The fact that $\pi$ is a submersion follows from \eqref{eq:vertical-projection-lower}. To see properness, let $K\subset\mathbb R^3$ be compact. Then
\begin{align*}
        (\pi|_{M})^{-1}(K)
        \subset
        \Sigma\cap\left(K\times[0,\tau]\right).
\end{align*}
Because $\Sigma$ is proper, the set on the right is compact. Since $M$ is closed in $\Sigma$ and $\pi$ is continuous, $(\pi|_{M})^{-1}(K)$ is a closed subset of this compact set. Thus $\pi|_{M}$ is proper.

Finally, each connected component of $\partial M$ lies in $\{h=\tau\}$. Since $h$ is constant on each such component, $\pi$ is injective on it. Hence \cite[Lemma~1.4]{meeksRosenberg2005uniqueness} implies that $\pi|_M$ is injective. Therefore $\pi|_V$ is a diffeomorphism onto $D$.

We next identify the image of $M$. Since $M=\overline V$ and $\pi$ is continuous, $\pi(M)\subset\overline D.$ Conversely, let $y_j\in D$ converge to $y\in\overline D$, and choose
$p_j\in V$ such that $\pi(p_j)=y_j$. Let $K\subset\mathbb R^3$ be a
compact set containing the sequence $\{y_j\}$. Then
\begin{align}
        p_j
        \in
        \Sigma\cap\left(K\times[0,\tau]\right).
\end{align}
Since $\Sigma$ is proper, a subsequence converges to some
$p\in\Sigma$. As $p_j\in V$, we have $p\in M$, and continuity gives
$\pi(p)=y$. Consequently,
\begin{align}
        \pi(M)=\overline D.
        \label{eq:projection-closure}
\end{align}

Moreover, $\pi|_M:M\rightarrow\overline D$ is proper. This is because, if
$K\subset\overline D$ is compact, then $(\pi|_M)^{-1}(K)$ is a closed
subset of the compact set
\begin{align}
        \Sigma\cap\left(K\times[0,\tau]\right).
\end{align}
Thus $\pi|_M$ is a proper continuous bijection and therefore a
homeomorphism. Since $\pi|_V$ is a local diffeomorphism, $D$ is open.
Injectivity and \eqref{eq:M-boundary} give
\begin{align}
        \pi(\partial M)
        =
        \pi(M\setminus V)
        =
        \overline D\setminus D
        =
        \partial D.
        \label{eq:projection-boundary}
\end{align}

For $p\in\partial M$, equation
\eqref{eq:vertical-projection-lower} and the inverse function theorem
give neighborhoods $U_p\subset\Sigma$ of $p$ and
$O_p\subset\mathbb R^3$ of $\pi(p)$ such that
\begin{align}
        \pi|_{U_p}:U_p\rightarrow O_p
\end{align}
is a diffeomorphism. Since $(\pi|_M)^{-1}$ is continuous, after
shrinking $O_p$ we may assume that
\begin{align}
        (\pi|_M)^{-1}(\overline D\cap O_p)
        \subset
        U_p.
\end{align}
Consequently, the global inverse agrees on
$\overline D\cap O_p$ with the smooth local inverse
$(\pi|_{U_p})^{-1}$. Since $\pi(\partial M)=\partial D$, it also
follows that $\partial D$ is smooth.

Now define
\begin{align}
        u=h\circ(\pi|_M)^{-1}.
\end{align}
Then $u\in C^\infty(\overline D)$, $u=\tau$ on $\partial D$, and
\begin{align}
        \frac{|Du|^2}{1+|Du|^2}
        =
        |\grad_\Sigma h|^2
        \leq
        2\Lambda_*\tau
        <
        \frac12,
\end{align}
which proves \eqref{eq:low-component-slope}.
\end{proof}

Now we prove a weighted tilt integral as in Colding and Minicozzi \cite{colding2026minimal}. Eventually our goal is to show that this integral is infinite if we assume that \(H_2(\Sigma; \bZ_2) = 0\) and that \(\Sigma\) is not a plane. Choose a connected component \(V\) of \(\{h<\tau\}\) which meets \(\{h<\tau/4\}\), and let \(u\) and \(D\) be as in Lemma~\ref{lem:low-component-global-graph}. Then we have \begin{lem}\label{lem:weighted-base-energy}
Identify $\{x_4=0\}$ with $\mathbb R^3$ and define $g:\mathbb R^3\rightarrow\mathbb R$ by
\begin{align}
        g(y)=
        \begin{cases}
        \tau-u(y),&y\in D,\\
        0,&y\notin D.
        \end{cases}
\end{align}
Then $g$ is globally $1$-Lipschitz and
\begin{align}
        \int_{\mathbb R^3}
        \frac{|\grad g(y)|^2}{1+|y|}
        \,\ud y
        <\infty.
        \label{eq:finite-weighted-base-energy}
\end{align}
\end{lem}

\begin{proof}
Let $x:\Sigma\rightarrow\mathbb R^4$ denote the position vector and set
\begin{align}
        E=|\grad_\Sigma h|^2,
        \qquad
        \Phi=\frac12(\tau-h)^2,
        \qquad
        \psi(x)=\frac{1}{\sqrt{1+|x|^2}}.
\end{align}
Then
\begin{align}
        \Delta_\Sigma\Phi=E,
        \qquad
        \Phi=|\grad_\Sigma\Phi|=0
        \quad\text{on }\partial V.
        \label{eq:quadratic-height-test}
\end{align}
Since $\Sigma$ is minimal and three-dimensional,
\begin{align}
        \Delta_\Sigma\psi
        =
        -3\frac{1+|x^\perp|^2}{(1+|x|^2)^{5/2}}
        \leq0,
        \label{eq:psi-superharmonic}
\end{align}
Here $x^\perp$ denotes the normal component of the position vector $x$. For $R\geq1$, let $\eta_R\in C_c^\infty(B_{2R}^{\mathbb R^4}(0))$ be a standard ambient radial cutoff satisfying
\begin{align}
        0\leq\eta_R\leq1,
        \qquad
        \eta_R=1\text{ on }B_R^{\mathbb R^4}(0),
        \qquad
        |\grad\eta_R|\leq\frac{C}{R},
        \qquad
        |\grad^2\eta_R|\leq\frac{C}{R^2}.
\end{align}
We choose the cutoffs so that $\eta_R\rightarrow1$ pointwise, and set $q_R=\eta_R\psi$. The support of $q_R$ in $\overline V$ is compact by properness. Integration by parts and \eqref{eq:quadratic-height-test} give
\begin{align}
        \int_V q_R E\,\ud\mu_\Sigma
        =
        \int_V\Phi\Delta_\Sigma q_R\,\ud\mu_\Sigma.
        \label{eq:quadratic-weighted-identity}
\end{align}
There is no boundary term because both $\Phi$ and its normal derivative vanish on $\partial V$. On the transition annulus,
\begin{align}
        \psi\leq\frac{C}{R},
        \qquad
        |\grad_\Sigma\psi|\leq\frac{C}{R^2},
        \qquad
        |\Delta_\Sigma\eta_R|\leq\frac{C}{R^2}.
\end{align}
Consequently, \eqref{eq:psi-superharmonic} gives
\begin{align}
        \Delta_\Sigma q_R
        \leq
        \frac{C}{R^3}
        \mathbf 1_{B_{2R}^{\mathbb R^4}(0)\setminus B_R^{\mathbb R^4}(0)}.
\end{align}
Since $V$ is a graph with $|Du|<1$,
\begin{align}
        \mu_\Sigma(V\cap B_{2R}^{\mathbb R^4}(0))
        \leq
        CR^3.
        \label{eq:single-sheet-cubic-growth}
\end{align}
As $0\leq\Phi\leq\tau^2/2$, equation \eqref{eq:quadratic-weighted-identity} gives
\begin{align}
        \int_V\eta_R\psi E\,\ud\mu_\Sigma
        \leq C.
\end{align}
Letting $R\rightarrow\infty$ and applying Fatou's lemma yields
\begin{align}
        \int_V
        \frac{E}{\sqrt{1+|x|^2}}
        \,\ud\mu_\Sigma
        <\infty.
        \label{eq:finite-graph-tilt}
\end{align}

On the graph of $u$,
\begin{align}
        E
        =
        \frac{|Du|^2}{1+|Du|^2},
        \qquad
        \ud\mu_\Sigma
        =
        \sqrt{1+|Du|^2}\,\ud y.
\end{align}
Since $|Du|<1$ and $0<u<\tau$, there is a constant $c=c(\tau)>0$ such that
\begin{align}
        \frac{E}{\sqrt{1+|x|^2}}\,\ud\mu_\Sigma
        &=
        \frac{|Du|^2}
        {\sqrt{1+|Du|^2}\sqrt{1+|y|^2+u^2}}
        \,\ud y
        \\
        &\geq
        c\frac{|Du|^2}{1+|y|}\,\ud y.
\end{align}
Therefore \eqref{eq:finite-graph-tilt} implies
\begin{align}
        \int_D
        \frac{|Du(y)|^2}{1+|y|}
        \,\ud y
        <\infty.
        \label{eq:finite-Du-energy}
\end{align}

It remains to prove the Lipschitz estimate. If the segment joining $y,z\in D$ is contained in $D$, then \eqref{eq:low-component-slope} gives $|g(y)-g(z)|\leq|y-z|$. If it leaves $D$, assume that $g(y)\geq g(z)$ and let $\xi$ be the first point of $\partial D$ on the segment starting at $y$. Since $g(\xi)=0$,
\begin{align}
        |g(y)-g(z)|
        \leq
        g(y)
        =
        |g(y)-g(\xi)|
        \leq
        |y-\xi|
        \leq
        |y-z|.
\end{align}
The cases in which one or both endpoints lie outside $D$ follow in the same way. Thus $g$ is globally $1$-Lipschitz. Its weak gradient is $-Du$ almost everywhere on $D$ and zero almost everywhere outside $D$, so \eqref{eq:finite-weighted-base-energy} follows from \eqref{eq:finite-Du-energy}.
\end{proof}
\section{Proof of the Half-Space Theorem}
Our goal is now to show that the weighted integral \eqref{eq:finite-graph-tilt} is infinite on \(V\) assuming \(H_2(\Sigma; \bZ_2) = 0\). The assumption \(H_2(\Sigma; \bZ_2) = 0\) is crucial, as the tilt integral is finite on a catenoid \(\mathcal{C} \subset \bR^4\), which has \(\dim H_2(\mathcal{C};\bZ_2) = 1\). We proceed using a capacity estimate for spheres (cf. \cite[Equation 2.2.14]{maz2013sobolev}):
\begin{lem}\label{lem:half-space-two-cap}
There exist universal constants $c,C>0$ with the following property. Let $r>0$, set $\mathbb S_r^2=\partial B_r^{\mathbb R^3}(0)$, and let $0<\rho<r/100$, and let $f\in W^{1,2}(\mathbb S_r^2)$. Suppose that there are geodesic caps $\mathcal{C}_1,\mathcal{C}_2\subset\mathbb S_r^2$, both of radius $\rho$, such that
\begin{align}
        f\geq \beta
        \quad \text{a.e. on } \mathcal{C}_1,
        \quad
        f\leq \alpha
        \quad\text{a.e. on } \mathcal{C}_2,
        \label{eq:half-space-cap-values}
\end{align}
where $\beta>\alpha\geq 0.$ Then
\begin{align}
        \int_{\mathbb S_r^2}
        |\nabla_{\mathbb S_r^2}f|^2
        \,\ud\mathcal H^2
        \geq
        \frac{c(\beta-\alpha)^2}{\log(Cr/\rho)}.
        \label{eq:half-space-two-cap}
\end{align}
\end{lem}
\begin{proof}
By scaling, it suffices to prove \eqref{eq:half-space-two-cap} on the
unit sphere, with $r=1$. Let $\ud\sigma=\ud A/(4\pi)$, where $\ud A$
is the area form on the sphere $\mathbb S^2$. If
$\mathcal C\subset\mathbb S^2$ is a spherical cap of radius $\rho$,
define the normalized measure
$\ud\mu_{\mathcal C}=|\mathcal C|^{-1}\mathbf 1_{\mathcal C}\,\ud A$,
where $\mathbf 1_{\mathcal C}$ is the indicator function of
$\mathcal C$, and set
$\nu_{\mathcal C}=\mu_{\mathcal C}-\sigma$.

Since $r=1$ and $\rho<1/100$, we have $|\mathcal C|\simeq\rho^2$. Let $\phi_{\mathcal C}\in W^{1,2}(\mathbb S^2)$ be the unique mean-zero weak solution of
\begin{align}
        -\Delta_{\mathbb S^2}\phi_{\mathcal C}
        =
        \nu_{\mathcal C},
        \qquad
        \int_{\mathbb S^2}
        \phi_{\mathcal C}(x)\,\ud\sigma(x)
        =
        0.
        \label{eq:half-space-cap-potential-new}
\end{align}
We first prove the estimate
\begin{align}
        \int_{\mathbb S^2}
        |\nabla\phi_{\mathcal C}|^2
        \,\ud A
        \leq
        C\log\frac{C}{\rho}
        \label{eq:half-space-cap-potential-bound-new}
\end{align}
for some universal constant $C>0.$ Let $G(x,z)$ be the mean-zero Green kernel (cf. \cite[Theorem~4.13]{aubin1998some}) normalized by
\begin{align}
        -\Delta_xG(x,z)
        =
        \delta_z-\sigma,
        \quad
        \int_{\mathbb S^2}
        G(x,z)\,\ud\sigma(x)
        =
        0.
        \label{eq:half-space-sphere-Green}
\end{align}
Then, up to an additive constant, the Green's function satisfies (cf. \cite[Appendix~A]{okikiolu2008negative})
\begin{align}
        G(x,y)
        =
        -\frac{1}{4\pi}
        \log(1-\langle x,y\rangle),
\end{align}
and hence
\begin{align}
        |G(x,y)|
        \leq
        C\left(1+
        \left|\log d_{\mathbb S^2}(x,y)\right|\right),
        \label{eq:half-space-Green-log-bound-new}
\end{align}
for some universal constant $C>0$. Thus, we obtain
\begin{align}
        \int_{\mathbb S^2}
        |\nabla\phi_{\mathcal C}|^2
        \,\ud A
        =\int_{\mathbb S^2}
\phi_{\mathcal C}\,\ud\nu_{\mathcal C} =
        \iint_{\mathbb S^2\times\mathbb S^2}
        G(x,y)\,\ud\nu_{\mathcal C}(x)\,\ud\nu_{\mathcal C}(y)
        =
        \iint_{\mathcal C\times\mathcal C}
        G(x,y)\,\ud\mu_{\mathcal C}(x)\,\ud\mu_{\mathcal C}(y),\qquad 
        \label{eq:half-space-cap-energy-reduction-new}
\end{align}
where in the first equality we used
\eqref{eq:half-space-cap-potential-new} and integration by parts, in the second equality we used the Green's function representation, and in the third equality we used the mean-zero normalization \eqref{eq:half-space-sphere-Green} and the symmetry of the Green's function:
\begin{align}
        \int_{\mathbb S^2}
        G(x,z)\,\ud\sigma(x)
        =
        \int_{\mathbb S^2}
        G(x,z)\,\ud\sigma(z)
        =
        0.
\end{align}
For each fixed $y\in\mathcal C$, the spherical cap $\mathcal C$ is contained in a geodesic ball $B_{2\rho}^{\mathbb S^2}(y)$. Therefore, geodesic polar coordinates
centered at $y$ give
\begin{align}
        \int_{\mathcal C}
        \left(1+
        \left|\log d_{\mathbb S^2}(x,y)\right|\right)
        \,\ud A(x)
        \leq
        C\int_0^{2\rho}
        (1+|\log s|)\sin s\,\ud s                                  
        \leq
        C\rho^2\log\frac{C}{\rho}.
        \label{eq:half-space-log-cap-integral-new}
\end{align}
Combining \eqref{eq:half-space-Green-log-bound-new},
\eqref{eq:half-space-cap-energy-reduction-new}, and
\eqref{eq:half-space-log-cap-integral-new}, and using
$|\mathcal C|\simeq\rho^2$, gives
\eqref{eq:half-space-cap-potential-bound-new}. From \eqref{eq:half-space-cap-potential-new}, for every
$f\in W^{1,2}(\mathbb S^2)$,
\begin{align}
        \left|
        \int_{\mathcal C}f\,\ud\mu_{\mathcal C}
        -
        \int_{\mathbb S^2}f\,\ud\sigma
        \right|
        =
        \left|
        \int_{\mathbb S^2}f\,\ud\nu_{\mathcal C}
        \right|                                                       =
        \left|
        \int_{\mathbb S^2}
        \langle\nabla f,\nabla\phi_{\mathcal C}\rangle
        \,\ud A
        \right|                                                      
        \leq
        C\left(\log\frac{C}{\rho}\right)^{1/2}
        \left(
        \int_{\mathbb S^2}|\nabla f|^2\,\ud A
        \right)^{1/2}.
        \label{eq:half-space-cap-average-new}\quad 
\end{align}
Applying \eqref{eq:half-space-cap-average-new} to
$\mathcal C_1$ and $\mathcal C_2$, and using
\eqref{eq:half-space-cap-values}, we obtain
\begin{align}
        \beta-\alpha
        &\leq
        \int_{\mathcal C_1}f\,\ud\mu_{\mathcal C_1}
        -
        \int_{\mathcal C_2}f\,\ud\mu_{\mathcal C_2}\leq
        C\left(\log\frac{C}{\rho}\right)^{1/2}
        \left(
        \int_{\mathbb S^2}|\nabla f|^2\,\ud A
        \right)^{1/2}.
\end{align}
Squaring and rearranging the above display gives \eqref{eq:half-space-two-cap}.
\end{proof}
We use the following lemma to locate spherical caps on the set \(V\):
\begin{lem}\label{lem:two-noncompact-levels}
There is a regular value
\begin{align}
        t\in(\tau/4,\tau/2)
\end{align}
such that $\partial V$ has a connected noncompact component $S\subset\{h=\tau\}$ and $V\cap\{h=t\}$ has a connected noncompact component $\Gamma\subset\{h=t\}$.
\end{lem}

\begin{proof}
The boundary $\partial V$ is nonempty. Otherwise $V$ is both open and closed in $\Sigma$, and hence $V=\Sigma$. This would give $h<\tau$ on $\Sigma$, contrary to $\tau<\sup_\Sigma h$. Since $\tau$ is a regular value, every connected component of $\partial V$ is a connected component of the full level $\{h=\tau\}$. Lemma~\ref{lem:no-compact-regular-level} shows that each such component is noncompact. Fix one of them and denote it by $S$.

Since points of $V$ approach $S$, the image $h(V)$ contains $(\tau/4,\tau)$. Sard's theorem gives a regular value $t\in(\tau/4,\tau/2)$. In particular, $V\cap\{h=t\}$ is nonempty. Let $\Gamma$ be one of its connected components. Since $t<\tau$, it is also a connected component of the full level $\{h=t\}$, and Lemma~\ref{lem:no-compact-regular-level} implies that $\Gamma$ is noncompact.
\end{proof}

\begin{proof}[Proof of Theorem~\ref{thm: maintheorem}]
Let $\Sigma$ satisfy the assumptions of Theorem~\ref{thm: maintheorem} and suppose, for contradiction, that $\Sigma$ is not a plane. Let $V$, $S$, $\Gamma$, $t$, and $g$ be as above, and set
\begin{align}
        S_0=\pi(S),
        \qquad
        \Gamma_0=\pi(\Gamma).
\end{align}
The sets $S$ and $\Gamma$ are closed in $\overline V$. Since $\pi:\overline V\rightarrow\overline D$ is a homeomorphism, their images are closed, connected, and noncompact in $\overline D$. As $\overline D$ is closed in $\mathbb R^3$, the sets $S_0$ and $\Gamma_0$ are closed, connected, and unbounded in $\mathbb R^3$. Moreover,
\begin{align}
        g=0
        \quad\text{on }S_0,
        \qquad
        g=\tau-t
        \quad\text{on }\Gamma_0.
        \label{eq:two-trace-values}
\end{align}
The sets $\{|y|:y\in S_0\}$ and $\{|y|:y\in\Gamma_0\}$ are connected unbounded subsets of $[0,\infty)$. Hence there is $R_0<\infty$ such that $S_0$ and $\Gamma_0$ meet every sphere $\partial B_r^{\mathbb R^3}(0)$ with $r\geq R_0$.

Set
\begin{align}
        a=\tau-t,
        \qquad
        \rho_0=\frac{a}{4}.
\end{align}
For every $r\geq R_0$, choose
\begin{align}
        q_r\in S_0\cap\partial B_r^{\mathbb R^3}(0),
        \qquad
        p_r\in\Gamma_0\cap\partial B_r^{\mathbb R^3}(0).
\end{align}
Since $g$ is $1$-Lipschitz, on the geodesic cap of radius $\rho_0$ centered at $q_r$ we have
\begin{align}
        g\leq\frac{a}{4},
\end{align}
while on the geodesic cap of radius $\rho_0$ centered at $p_r$ we have
\begin{align}
        g\geq\frac{a}{2}.
\end{align}
After increasing $R_0$, Lemma~\ref{lem:half-space-two-cap}, applied to $f=g|_{\partial B_r^{\mathbb R^3}(0)}$ with
$\alpha=a/4$ and $\beta=a/2$, gives
\begin{align}
        \int_{\partial B_r^{\mathbb R^3}(0)}
        |\grad_{\partial B_r}g|^2
        \,\ud\mathcal H^2
        \geq
        \frac{ca^2}{16\log(Cr/\rho_0)}
\end{align}
for every $r\geq R_0$. Consequently,
\begin{align}
        \int_{\mathbb R^3}
        \frac{|\grad g(y)|^2}{1+|y|}
        \,\ud y
        &\geq
        \int_{R_0}^\infty
        \frac{1}{1+r}
        \int_{\partial B_r^{\mathbb R^3}(0)}
        |\grad_{\partial B_r}g|^2
        \,\ud\mathcal H^2
        \,\ud r
        \\
        &\geq
        \frac{ca^2}{16}
        \int_{R_0}^\infty
        \frac{\ud r}{(1+r)\log(Cr/\rho_0)}
        =
        \infty.
\end{align}
This contradicts \eqref{eq:finite-weighted-base-energy}. Thus the assumption that \(\Sigma\) is not a plane is impossible, and hence \(\Sigma\) is a plane.
\end{proof}
\section{Proof of the Slab Theorem}
\label{subsec:slab-without-curvature}
In this section we will prove Theorem~\ref{thm:slab-without-curvature}. The key point here is that we do not assume bounded curvature, as in the previous section. Instead, we exploit the fact that properly embedded minimal hypersurfaces confined to slabs have
Euclidean volume growth by a recent result of Colding and Minicozzi \cite{colding2026minimal}.

After a rigid motion and a dilation, we may assume that
$\Sigma\subset\bR^3\times[0,1]$. Set
\begin{align}
        h=x_4|_\Sigma,
        \qquad
        E=|\grad_\Sigma h|^2.
\end{align}
By \cite[Theorems~0.4 and~4.2]{colding2026minimal}, there is a constant
$C_0<\infty$ such that
\begin{align}
        \mu_\Sigma(\Sigma\cap B_R(0))\leq C_0R^3
        \quad(R\geq1),
        \qquad
        \int_\Sigma\frac{E}{1+|x|}\,\ud\mu_\Sigma<\infty.
        \label{eq:slab-volume-tilt}
\end{align}
Define
\begin{align}
        F(p)=\int_\Sigma
        \frac{E(x)}{1+|x-p|}\,\ud\mu_\Sigma(x),
        \qquad p\in\bR^4.
        \label{eq:slab-potential-definition}
\end{align}
We first prove an estimate for $F$.

\begin{lem}\label{lem:slab-potential}
The function $F$ is finite everywhere. If $\Omega\subset\bR^4$ is connected
and unbounded, then
\begin{align}
        \inf_{p\in \Omega}F(p)=0.
        \label{eq:slab-potential-infimum}
\end{align}
\end{lem}

\begin{proof}
Let $x\in \Sigma$ and $p\in \bR^4$. Since $1+|x|\leq(1+|p|)(1+|x-p|),$ the finiteness of $F(p)$ follows from \eqref{eq:slab-volume-tilt}.
Moreover, \eqref{eq:slab-volume-tilt}, \(E \le 1\), and dominated convergence give
\begin{align}
        \lim_{r\to \infty}\frac{1}{r}\int_{\Sigma\cap\{|x|\leq r\}}
        E\,\ud\mu_\Sigma =\int_{\Sigma}\left(\lim_{r\to \infty}\frac{E\cdot \mathbf{1}_{\{|x| \le r\}}}{r}\right)
        \,\ud\mu_\Sigma = 0.
        \label{eq:slab-tilt-sublinear}
\end{align}
Suppose, to the contrary, there exists $c>0$ such that $F\geq c>0$ on $\Omega$. Since $\Omega$ is connected and unbounded, $\Omega \cap \mathbb{S}^3_r\not =\emptyset$ for large enough radius $r\gg 1$. For large enough $r\gg 1$, set $N=\lfloor r\rfloor$ and
choose
\begin{align}
        p_i\in \Omega,
        \quad |p_i|=2r+i,
        \quad 1\leq i\leq N.
\end{align}
The reverse triangle inequality gives
\begin{align}
        \sum_{i=1}^N\frac{1}{1+|x-p_i|}
        &\leq
        \sum_{i=1}^N\frac{1}{1+\bigl||x|-(2r+i)\bigr|}
        \leq C\log r,
\end{align}
for some universal constant $C>0$ independent of $r$. Furthermore, $|x-p_i|\geq r$ when $|x|\leq r$, and
$|x-p_i|\geq |x|/2$ when $|x|\geq6r$. Consequently,
\begin{align}
        c\leq\frac1N\sum_{i=1}^NF(p_i)
        &\leq
        \frac1r\int_{\Sigma\cap\{|x|\leq r\}}E\,\ud\mu_\Sigma
        +\frac{C\log r}{r}
        \int_{\Sigma\cap\{r<|x|<6r\}}E\,\ud\mu_\Sigma +
        2\int_{\Sigma\cap\{|x|\geq6r\}}
        \frac{E}{|x|}\,\ud\mu_\Sigma.
        \label{eq:slab-potential-annuli}
\end{align}
By \eqref{eq:slab-volume-tilt} and \eqref{eq:slab-tilt-sublinear},
the first and last terms on the right hand side tend to zero as $r\to \infty$. Thus, for all sufficiently
large $r$,
\begin{align}
        \int_{\Sigma\cap\{r<|x|<6r\}}
        \frac{E}{1+|x|}\,\ud\mu_\Sigma
        \geq\frac{c'}{\log r}
\end{align}
for some $c'>0$ independent of $r>0$. Taking $r=8^j$ for large enough integers $j\gg 1$ gives disjoint annuli, and summing over $j$ contradicts \eqref{eq:slab-volume-tilt}.
\end{proof}
We now compare averages of functions of the height over small and
large balls.
\begin{lem}\label{lem:slab-weighted-mean-value}
Let $\psi\in C_c^\infty(\bR)$ satisfy $0\leq\psi\leq1$, and set
\begin{align}
        A_\psi=
        \liminf_{R\rightarrow\infty}
        \frac1{R^3}
        \int_{\Sigma\cap B_R(0)}\psi(h)\,\ud\mu_\Sigma.
        \label{eq:slab-weighted-volume-limit}
\end{align}
Then, for every $p\in\Sigma$ and $\delta>0$,
\begin{align}
        A_\psi
        \geq
        \frac1{\delta^3}
        \int_{\Sigma\cap B_\delta(p)}\psi(h)\,\ud\mu_\Sigma
        -\frac{\|\psi''\|_\infty}{2}
        (1+\delta^{-1})F(p).
        \label{eq:slab-weighted-mean-value}
\end{align}
\end{lem}
\begin{proof}
Fix $p\in\Sigma$ and $\delta>0$, and set
\begin{align}
        r(x)=|x-p|,
        \qquad
        f=\psi(h),
        \qquad
        V_f(s)=\int_{\Sigma\cap B_s(p)}f\,\ud\mu_\Sigma.
\end{align}
We also define
\begin{align}
        Q(s)=\frac{V_f(s)}{s^3},
        \qquad
        C_\psi=\frac{\|\psi''\|_\infty}{2}.
\end{align}
Since $f\geq0$, the inclusions $B_{R-|p|}(p)\subset B_R(0)\subset B_{R+|p|}(p)$ for $|p|<R$ give
\begin{align}
        \left(\frac{R-|p|}{R}\right)^3Q(R-|p|)
        \leq
        \frac1{R^3}
        \int_{\Sigma\cap B_R(0)}f\,\ud\mu_\Sigma \leq
        \left(\frac{R+|p|}{R}\right)^3Q(R+|p|).
\end{align}
Taking the $\liminf$ as $R\to\infty$, we obtain
\begin{align}
        A_\psi=\liminf_{R\to\infty}Q(R).
        \label{eq:slab-center-independent}
\end{align}

We next derive a differential inequality for $Q$. Set
\begin{align}
        Y=(x-p)^T
        =\grad_\Sigma\left(\frac{r^2}{2}\right)
        =r\grad_\Sigma r.
\end{align}
Since $\Sigma^3\subset \bR^4$ is minimal we have
\begin{align}
        \operatorname{div}_\Sigma Y =\Delta_\Sigma\left(\frac{r^2}{2}\right) =\sum_{i=1}^4|\grad_\Sigma x_i|^2 +\sum_{i=1}^4(x_i-p_i)\Delta_\Sigma x_i =3.
\end{align}
Let $s>0$ be a regular value of $r|_\Sigma$, and write $ \Omega_s=\Sigma\cap B_s(p).$ Properness implies that $\overline{\Omega_s}$ is compact.
Let $\eta={\grad_\Sigma r}/{|\grad_\Sigma r|}$ denote the outward unit conormal to $\partial\Omega_s$ in $\Sigma$. Since $r=s$ on $\partial\Omega_s$, $ Y\cdot\eta=s|\grad_\Sigma r|.$ The coarea formula, together with $f\geq0$ and
$|\grad_\Sigma r|\leq1$, gives
\begin{align}
        sV_f'(s)
        =s\int_{\partial\Omega_s}
        \frac{f}{|\grad_\Sigma r|}\,\ud\mathcal H^2 \geq s\int_{\partial\Omega_s}
        f|\grad_\Sigma r|\,\ud\mathcal H^2.
\end{align}
Applying the divergence theorem to $fY$, we obtain
\begin{align}
        s\int_{\partial\Omega_s}
        f|\grad_\Sigma r|\,\ud\mathcal H^2
        =\int_{\partial\Omega_s}fY\cdot\eta\,\ud\mathcal H^2
        =\int_{\Omega_s}\operatorname{div}_\Sigma(fY)
        \,\ud\mu_\Sigma
        =3V_f(s)+\int_{\Omega_s}
        \grad_\Sigma f\cdot Y\,\ud\mu_\Sigma.
\end{align}
To further simplify the last integral, set $\phi_s=\frac{s^2-r^2}{2}.$ Then $\phi_s=0$ on $\partial\Omega_s$ and $\grad_\Sigma\phi_s=-Y$. Integration by parts therefore gives
\begin{align}
        \int_{\Omega_s}\phi_s\Delta_\Sigma f\,\ud\mu_\Sigma
        =\int_{\partial\Omega_s}
        \phi_s\,\partial_\eta f\,\ud\mathcal H^2
        -\int_{\Omega_s}
        \grad_\Sigma\phi_s\cdot\grad_\Sigma f\,\ud\mu_\Sigma
        =\int_{\Omega_s}
        Y\cdot\grad_\Sigma f\,\ud\mu_\Sigma.
\end{align}
Combining the preceding identities gives, for almost every $s>0$,
\begin{align}
        sV_f'(s)
        \geq
        3V_f(s)+\frac12
        \int_{\Omega_s}
        (s^2-r^2)\Delta_\Sigma f\,\ud\mu_\Sigma.
        \label{eq:slab-weighted-radial-identity}
\end{align}
Since $h$ is harmonic and $E=|\grad_\Sigma h|^2$,
\begin{align}
        \Delta_\Sigma f
        =\psi''(h)|\grad_\Sigma h|^2+\psi'(h)\Delta_\Sigma h
        =\psi''(h)E.
\end{align}
Thus, using $0\leq s^2-r^2\leq s^2$ on $\Omega_s$, we obtain
\begin{align}
        Q'(s)
        &=\frac{sV_f'(s)-3V_f(s)}{s^4}
        \geq
        \frac1{2s^4}
        \int_{\Omega_s}(s^2-r^2)\psi''(h)E\,\ud\mu_\Sigma
        \geq
        -\frac{C_\psi}{s^2}
        \int_{\Omega_s}E\,\ud\mu_\Sigma.
\end{align}
For $R>\delta$, integrating this inequality gives
\begin{align}\label{eqn:integral-inequality}
        Q(R)\geq Q(\delta)-C_\psi I_R,
\end{align}
where
\begin{align}
        I_R=
        \int_\delta^R\frac1{s^2}
        \left(
        \int_{\Sigma\cap B_s(p)}E(x)\,\ud\mu_\Sigma(x)
        \right)\ud s.
\end{align}
Since $E\geq0$, we can re-write the integral above as
\begin{align}
        I_R
        &=
        \int_\delta^R\int_\Sigma
        \frac{E(x)}{s^2}\mathbf1_{\{r(x)<s\}}
        \,\ud\mu_\Sigma(x)\,\ud s
        =
        \int_{\Sigma} E(x)
        \left(
        \int_\delta^R
        \frac{\mathbf1_{\{r(x)<s\}}}{s^2}\,\ud s
        \right)\ud\mu_\Sigma(x) \\
        &= \int_{\Sigma\cap B_R(p)} E(x)
        \left(
        \int_{\max\{\delta,r(x)\}}^R\frac{\ud s}{s^2}
        \right)\ud\mu_\Sigma(x) =\int_{\Sigma\cap B_R(p)}
E(x)\left(
    \frac{1}{\max\{\delta,r(x)\}}-\frac{1}{R}
\right)\,\ud\mu_\Sigma(x).
\end{align}
Note that as $R\to \infty$ we have
\begin{align}
        I_\infty:= \lim_{R\to\infty} I_R
        =
        \int_\Sigma
        \frac{E(x)}{\max\{\delta,|x-p|\}}
        \,\ud\mu_\Sigma(x).
\end{align}
Here $I_\infty$ is finite since
\begin{align}
        \frac{1+r(x)}{\max\{\delta,r(x)\}}
        =
        \frac1{\max\{\delta,r(x)\}}
        +\frac{r(x)}{\max\{\delta,r(x)\}}
        \leq\delta^{-1}+1,
\end{align}
and therefore 
\begin{align}
        I_\infty \leq
        (1+\delta^{-1})
        \int_\Sigma\frac{E(x)}{1+|x-p|}\,\ud\mu_\Sigma(x)
        =(1+\delta^{-1})F(p)<\infty.
\end{align}
Taking the $\liminf$ in the inequality \eqref{eqn:integral-inequality} and using \eqref{eq:slab-center-independent} we get
\begin{align}\label{eqn:estimate}
        A_\psi
        & \geq
        \frac{V_f(\delta)}{\delta^3}
        -\frac{\|\psi''\|_\infty}{2}
        \int_\Sigma
        \frac{E(x)}{\max\{\delta,|x-p|\}}\,\ud\mu_\Sigma(x)
        \\ &\geq
        \frac1{\delta^3}
        \int_{\Sigma\cap B_\delta(p)}\psi(h)\,\ud\mu_\Sigma
        -\frac{\|\psi''\|_\infty (1+\delta^{-1})F(p)}{2},
\end{align}
which implies \eqref{eq:slab-weighted-mean-value}.
\end{proof}
We are now ready to prove Theorem~\ref{thm:slab-without-curvature}. 
\begin{proof}[Proof of Theorem~\ref{thm:slab-without-curvature}]
Suppose, for contradiction, that $h$ is nonconstant. Let $\omega_3$
denote the volume of the unit ball in $\bR^3$, and choose an integer
$m>C_0/\omega_3$. Since $h(\Sigma)$ contains an interval, Sard's
theorem gives distinct regular values $t_1,\ldots,t_m\in\inte h(\Sigma).$ Choose functions $\psi_i\in C_c^\infty(\bR)$ with pairwise disjoint
supports such that $0\leq\psi_i\leq1$ and $\psi_i=1$ near $t_i$.
For each integer $1\leq i\leq m$, choose $\delta_i>0$ such that $\psi_i=1\quad\text{on }[t_i-\delta_i,t_i+\delta_i].$

Let $\Gamma_i$ be a connected component of $\{h=t_i\}$.
Lemma~\ref{lem:compact-height-components} and
$H_2(\Sigma;\bZ_2)=0$ imply that $\Gamma_i$ is noncompact. Since
$\Gamma_i$ is closed in $\Sigma$ and $\Sigma$ is proper, $\Gamma_i$
is an unbounded connected subset of $\bR^4$. Thus
Lemma~\ref{lem:slab-potential} gives
\begin{align}
        \inf_{p\in\Gamma_i}F(p)=0.
        \label{eq:slab-level-potential-infimum}
\end{align}
For every $p\in\Gamma_i$ and $x\in\Sigma\cap B_{\delta_i}(p)$,
we have $|h(x)-t_i|\leq|x-p|<\delta_i$. Therefore, monotonicity gives
\begin{align}
        \frac1{\delta_i^3}
        \int_{\Sigma\cap B_{\delta_i}(p)}
        \psi_i(h)\,\ud\mu_\Sigma
        =\frac{\mu_\Sigma(\Sigma\cap B_{\delta_i}(p))}{\delta_i^3}
        \geq\omega_3.
\end{align}
Applying Lemma~\ref{lem:slab-weighted-mean-value} and then
\eqref{eq:slab-level-potential-infimum}, we obtain
\begin{align}
        A_{\psi_i}\geq\omega_3,
        \qquad 1\leq i\leq m.
\end{align}
Since the supports of the $\psi_i$ are disjoint,
$\sum_{i=1}^m\psi_i\leq1$. Consequently,
\begin{align}
        m\omega_3
        &\leq\sum_{i=1}^mA_{\psi_i}
        \leq
        \liminf_{R\rightarrow\infty}
        \frac1{R^3}
        \int_{\Sigma\cap B_R(0)}
        \sum_{i=1}^m\psi_i(h)\,\ud\mu_\Sigma \leq C_0,
\end{align}
contrary to the choice of the integer $m>C_0/\omega_3$. Thus $h$ is constant, which by the strong maximum principle implies that $\Sigma$ is a hyperplane.
\end{proof}

\section{Proof of Finite Total Curvature Equivalence}
We start by showing that if \(\Sigma\) is in a half-space and is not a plane, then \(h: \Sigma \rightarrow \bR\) must have a level set with a compact component.
\begin{lem}
\label{lem:compact-height-level}
Let $\Sigma\subset\bR^4$ be a connected, complete, properly embedded
minimal hypersurface with bounded curvature, and suppose that $\Sigma$ is
contained in a horizontal half-space. Then either $\Sigma$ is a horizontal
hyperplane, or there is a regular value $t$ of \(h\) such that
$\Sigma\cap\{h=t\}$ has a compact connected component.
\end{lem}
\begin{proof}
If $h$ is constant, then $\Sigma$ is a horizontal hyperplane.
Otherwise, suppose that no regular level of $h$ has a compact
connected component. Every connected component of every
regular level is then noncompact. Thus the argument of
Lemma~\ref{lem:two-noncompact-levels} and the proof of
Theorem~\ref{thm: maintheorem} apply, yielding a contradiction.
\end{proof}

\begin{lem}\label{lem:compact-angle-zero-set}
Let $\Sigma^3\subset\bR^3\times[0,1]$ be a complete, connected,
embedded minimal hypersurface with bounded curvature and
$\dim_{\bZ_2}H_2(\Sigma;\bZ_2)<\infty$. Fix a global unit normal $\nu$ and set
\[
w=\langle\nu,e_4\rangle,
\qquad
Z=\{w=0\}.
\]
Then $Z$ is compact.
\end{lem}

\begin{proof}
By \cite{aryanMcWeeneyTinaglia2026properness}, $\Sigma$ is proper.
If $Z$ is empty, there is nothing to prove. We may therefore
assume that $Z$ is nonempty. Since a hyperplane contained
in the slab is horizontal, $\Sigma$ is not a hyperplane
and $h$ is nonconstant. 

Now suppose for the sake of contradiction that \(Z\) is noncompact. Since \(\Sigma\) is contained in a slab and is proper, \(\pi|_\Sigma\) is proper. Thus \(\pi(Z)\) is unbounded. We may thus choose \(p_j \in Z\) with \(a_j = \pi(p_j)\) and \(|a_j| \ge 2^j\). Now set
\begin{equation}
    \Sigma_j := \Sigma - (a_j, 0), \qquad q_j := p_j - (a_j, 0) = (0, s_j).
\end{equation}
Since \(\Sigma\) is in a slab, \(s_j \in (0, 1)\) for all \(j\). Since \(\Sigma\) has bounded curvature and \(e_4 \in T_{p_j}\Sigma\) for every \(j\), there is furthermore \(\veps>0\) such that \(s_j \in [\veps, 1 - \veps]\) for every \(j\); indeed, Lemma~\ref{lem:height-gradient} shows that \(|\nabla_\Sigma h| < 1\) at points close to a boundary plane, but \(|\nabla_\Sigma h(p_j)| = 1\) for all \(j\). Passing to a subsequence we can now assume that
\begin{equation}
    q_j \rightarrow q_\infty := (0, s_\infty), \qquad s_\infty \in (0, 1).
\end{equation}
Now by the volume estimate of Colding and Minicozzi for properly embedded complete minimal surfaces in slabs
\cite[Theorem~0.8]{colding2026minimal}, there is a radius \(R_0 < \infty\) and a positive integer \(N < \infty\) such that when \(R \ge R_0\) and \(a \in \bR^3\), we have
\begin{equation}
    \mu_\Sigma(B_R(a, 0)) \le N\omega_3R^3.
\end{equation}
It follows that
\begin{equation}\label{eq:translated-local-volume}
    \sup_j \mu_{\Sigma_j}(B_R(0)) \le N\omega_3R^3 < \infty
\end{equation}
for every \(R < \infty\). By \eqref{eq:translated-local-volume} and \cite[Theorem~42.7 and Remark~42.8]{simon1983lectures}, a subsequence converges as integral varifolds. The uniform curvature bound and interior elliptic estimates give smooth convergence with finite integer multiplicity
on compact sets to a smooth embedded minimal hypersurface,
possibly disconnected,
\begin{equation}
\mathcal{L}\subset\bR^3\times[0,1].
\end{equation}
By \eqref{eq:translated-local-volume}, the hypersurface \(\mathcal{L}\) is in fact proper. This is because, for \(r>0\) smaller than the graphical radius of \(\Sigma\), \(B_r(x) \cap \mathcal{L}\) must consist of graphical sheets. By \eqref{eq:translated-local-volume}, there can only be finitely many sheets. It now follows that if \(\{x_i\} \subset \mathcal{L}\) converges to \(x \in \bR^4\), then infinitely many of the \(\{x_i\}\) lie in one graphical sheet of \(B_r(x) \cap \mathcal{L}\). Thus \(x \in \mathcal{L}\).

Now observe that \(q_\infty \in \mathcal{L}\), and let $L$ be the connected component of $\mathcal L$
containing $q_\infty$. Smooth convergence tells us that 
\begin{equation}
    e_4 \in T_{q_\infty}L
\end{equation}
since \(e_4 \in T_{q_j}\Sigma_j\) for all \(j\). It follows that \(L\) is not horizontal. Thus by Lemma~\ref{lem:compact-height-level}, there is a regular value \(t_0\) of \(h|_L\) and a compact connected component
\begin{equation}
    \Gamma_\infty \subset L \cap \{h = t_0\}.
\end{equation}
Since \(t_0\) is regular, we can flow \(\Gamma_\infty\) by \(X = \nabla_Lh / |\nabla_L h|^2\) to obtain an embedded compact collar neighborhood
\begin{equation}
    F: \Gamma_\infty \times [-3\veps, 3\veps] \rightarrow L
\end{equation}
for some sufficiently small \(\veps > 0\) satisfying \(h(F(x, s)) = t_0 + s\). Now choose a regular value \(t \in (t_0 - \veps, t_0 + \veps)\) of \(h|_\Sigma\). Note that \(t\) is also a regular value of \(\Sigma_j\) for every \(j\) since horizontal translations preserve the rank of \(dh\).

Since $L$ is properly embedded, it is two-sided. Choose a tubular neighborhood of the compact collar $F(\Gamma_\infty\times[-3\veps,3\veps])$ meeting no other connected component of $\mathcal L$. Smooth convergence with finite
multiplicity gives, over the smaller collar $F(\Gamma_\infty\times[-2\veps,2\veps])$, finitely many disjoint normal graphs in $\Sigma_j$. Ordering these graphs by signed normal height gives globally defined sheets. Choosing one sheet yields an embedding
\begin{equation}
F_j:\Gamma_\infty\times[-2\veps,2\veps]\longrightarrow\Sigma_j.
\end{equation} 
Furthermore, \(F_j\) converges to \(F\) in \(C^\infty\). Thus for small enough \(\veps\) and large enough \(j\) we have
\begin{equation}
    \frac{\partial}{\partial s}h(F_j(x, s)) > \frac{1}{2}
\end{equation}
Thus, the implicit function theorem gives a compact connected surface
\begin{equation}
    \Gamma_j \subset F_j(\Gamma_\infty \times [-2\veps, 2\veps]) \cap \{h = t\}
\end{equation}
which is diffeomorphic to \(\Gamma_\infty\). Note that \(\Gamma_j\) is a full connected component of \(\Sigma_j \cap \{h = t\}\) for large enough \(j\). This is because, if not then \(\Gamma_j\) is part of a component of \(\{h = t\}\) which crosses the boundary of \(F_j(\Gamma_\infty \times [-2\veps, 2\veps])\). The boundary components have height uniformly close to either \(t_0 - 2\veps\) or \(t_0 + 2\veps\), and thus the entire boundary is uniformly far from \(t\). Thus a connected component of \(\{h = t\}\) cannot intersect the boundary of \(F_j(\Gamma_\infty \times [-2\veps, 2\veps])\), and so \(\Gamma_j\) must be a full connected component.

Now set \(\widehat \Gamma_j = \Gamma_j + (a_j, 0)\). Then \(\widehat \Gamma_j \subset \Sigma\) is a compact level set of \(\{h = t\}\) in \(\Sigma\). Since $F_j\to F$ on a compact domain, there is $R<\infty$ such that $\Gamma_j\subset B_R(0)$ for all sufficiently
large $j$. As $|a_j|\to\infty$, we may pass to a subsequence
for which the balls $B_R((a_j,0))$ are pairwise disjoint. Thus the components $\widehat\Gamma_j$ are pairwise disjoint. But this contradicts Lemma~\ref{lem:compact-height-components}, which says that \(\{h = t\}\) has only finitely many compact components. This contradiction shows that \(Z\) is compact.
\end{proof}

\begin{proof}[Proof of Theorem~\ref{thm:main2}]
For the reverse implication, normalize the slab to
$\bR^3\times[0,1]$. By
\cite{aryanMcWeeneyTinaglia2026properness}, $\Sigma$ is proper,
and Lemma~\ref{lem:compact-angle-zero-set} shows that $Z$
is compact. Choose $R$ such that
$\pi(Z)\subset B_R^{\bR^3}(0)$, and set
$\Omega=\bR^3\setminus\overline{B_R^{\bR^3}(0)}$.
The restriction $\pi:\Sigma\cap\pi^{-1}(\Omega)\longrightarrow\Omega$ is a proper local diffeomorphism and hence a finite covering.
Since $\Omega$ is simply connected, its sheets are finitely
many global graphs. Moreover,
$\Sigma\cap\pi^{-1}(\overline{B_R^{\bR^3}(0)})$ is compact,
so these graphs contain all the ends of $\Sigma$.
Their heights are bounded, and interior gradient estimates
for the minimal surface equation give bounded slopes outside
a larger ball. Thus \cite[Proposition~3]{schoen1983uniqueness}
implies that $\Sigma$ is regular at infinity.
Each end has an expansion of the form
\[
    u(x)=a+b|x|^{-1}+O(|x|^{-2}),
    \qquad
    |D^ku(x)|=O(|x|^{-1-k})\quad(k\geq1).
\]
Thus \(|A_\Sigma(x)| \le C|D^2u(x)|\) decays like \(O(|x|^{-3})\), and finite total curvature follows.
\end{proof}
We are now ready to prove Corollary~\ref{cor:catenoid}.
\begin{proof}[Proof of Corollary~\ref{cor:catenoid}]
Suppose that $\dim_{\bZ_2}H_2(\Sigma;\bZ_2)=1$.
After a rigid motion and dilation, assume that $\Sigma\subset \bR^3 \times [0,1]$.
By Theorem~\ref{thm:main2} and its proof, $\Sigma$ is proper
and regular at infinity. Let $r$ be its number of ends.
There is a compact connected smooth domain $K\subset\Sigma$
such that $\Sigma\setminus\inte K$ consists of $r$ ends,
each diffeomorphic to $\bS^2\times[0,\infty)$, with
$\partial K=\Gamma_1\sqcup\cdots\sqcup\Gamma_r$ and
$\Gamma_i\simeq\bS^2$. Then using \cite[Proposition~9.1.2]{martelliGeometricTopology} we get
\[
r-1\leq\dim_{\bZ_2}H_2(K;\bZ_2)
=\dim_{\bZ_2}H_2(\Sigma;\bZ_2)
=1,
\]
which implies that $r\leq2$. If $r=1$, regularity at infinity implies that the height function $h=x_4|_\Sigma$ converges to a constant $a\in \bR$ at infinity. Since $h$ is harmonic, the maximum principle on an exhaustion by compact domains gives $h\equiv a$. Hence $\Sigma$ is a hyperplane, contradicting $\dim_{\bZ_2}H_2(\Sigma;\bZ_2)=1$. Therefore, $r=2$. Since $\Sigma$ is connected, \cite[Theorem~3]{schoen1983uniqueness} implies that $\Sigma$ is a three-dimensional catenoid up to rigid motions and dilations.

Conversely, the three-dimensional catenoid is diffeomorphic to
$\bS^2\times\bR$, so its second homology with
$\bZ_2$-coefficients is isomorphic to $\bZ_2$.
\end{proof}

\bibliographystyle{alpha}
\bibliography{refs}

\end{document}